\documentclass{amsart}

\usepackage{amscd}
\usepackage{amssymb}

\usepackage[T1]{fontenc}
\usepackage[latin1]{inputenc}
\usepackage{hyperref}
\usepackage[arrow,curve,matrix]{xy}
\usepackage{times}
\usepackage{graphicx}
\usepackage{color}
\usepackage{flafter}
\usepackage{placeins}
\usepackage{enumitem}

\usepackage[dpi=1270,PostScript=dvips,heads=LaTeX]{diagrams}
\diagramstyle{small,labelstyle=\scriptstyle}

\def\ZZ{{\mathbf Z}}

\def\CC{{\mathbf C}}

\def\OO{{\mathcal O}}
\def\R{\mathbf{R}}
\def\L{\mathbf{L}}
\def\SS{\mathcal{S}}

\def\D{\mathbf{D}}

\def\F{\mathcal{F}}
\def\E{\mathcal{E}}
\def\P{\mathcal{P}}

\def\G{\mathcal{G}}

\def\Pic0{{\rm Pic}^0(X)}

\newcommand{\alb}{\textnormal{alb}}
\newcommand{\tn}[1]{\textnormal{#1}}
\newcommand{\Alb}{\textnormal{Alb}}
\newcommand{\lra}{\longrightarrow}
\newcommand{\noi}{\noindent}

\newcommand{\RR}{\mathbf{R}}

\newcommand{\frakm}{\mathfrak{m}}

\newcommand{\HH}[3]{H^{{#1}} \big( {#2} , {#3}
\big) }

\newcommand{\codim}{\tn{codim}\, }

\theoremstyle{plain}

\newtheorem{theorem}{Theorem}[section]
\newtheorem{theoremalpha}{Theorem}

\newtheorem{proposition/example}[theorem]{Proposition/Example}
\newtheorem{proposition}[theorem]{Proposition}
\newtheorem{corollary}[theorem]{Corollary}
\newtheorem{lemma}[theorem]{Lemma}

\theoremstyle{definition}
\newtheorem{definition}[theorem]{Definition}
\newtheorem{remark}[theorem]{Remark}
\newtheorem{example}[theorem]{Example}

\newtheorem{conjecture/question}[theorem]{Conjecture/Question}

\newtheorem{remark/definition}[theorem]{Remark/Definition}
\newtheorem{definition/notation}[theorem]{Definition/Notation}

 \theoremstyle{remark}

\numberwithin{equation}{section}

\newcommand{\coker}{\operatorname{coker}}
\newcommand{\im}{\operatorname{im}}

\begin{document}

\title[Canonical cohomology as an exterior module]{Canonical cohomology as an exterior module}

\author {Robert Lazarsfeld}
\address{Department of Mathematics, University of Michigan,
525 E. University, Ann Arbor, MI 48109, USA } \email{{\tt
rlaz@umich.edu}}
\thanks{First author partially supported by NSF grant DMS-0652845}

\author{Mihnea Popa}
\address{Department of Mathematics, University of Illinois at Chicago,
851 S. Morgan Street, Chicago, IL 60607, USA } \email{{\tt
mpopa@math.uic.edu}}
\thanks{Second author partially supported by NSF grant DMS-0758253}

\author{Christian Schnell}
\address{Department of Mathematics, University of Illinois at Chicago,
851 S. Morgan Street, Chicago, IL 60607, USA } \email{{\tt
cschnell@math.uic.edu}}

\maketitle
 

{\dedicatory{\centerline{ \it To the memory of Eckart Viehweg}}}

\section*{Introduction}
Let $X$ be a smooth projective complex variety of dimension $d$, and set
\[
P_X \ = \ \bigoplus_{i = 0}^d \HH{i}{X}{\OO_X}  \ \ , \ \  Q_X \ = \  \bigoplus_{i = 0}^d \HH{i}{X}{\omega_X}.
\]
Via cup product, we may view these as graded modules over the exterior algebra
\[
E \ =_{\text{def}} \ \Lambda^{\bullet} \HH{1}{X}{\OO_X}.\footnote{Following the degree conventions of \cite{efs}, we take $E$ to be generated in degree $-1$, and we declare that the summand $\HH{i}{X}{\omega_X}$ of $Q_X$ has degree $-i$, whereas $\HH{i}{X}{\OO_X}$ has degree $d - i$ in $P_X$. Thanks to Serre duality, $Q_X$ and $P_X$ then become dual $E$-modules.}
\]
The ``complexity" of these $E$-modules was studied in \cite{lp}, where it was shown that while $P_X$ can behave rather unpredictably, the $E$-module $Q_X$ has quite simple algebraic properties. Specifically, let 
\[
k \ = \ k(X) \ = \ \dim X \, - \, \dim \alb_X(X)
\]
 denote the dimension of the generic fiber of the Albanese mapping \[\alb_X : X \lra \Alb(X)\]
 over its image. It was established in \cite{lp} (for compact K\"ahler manifolds) that $Q_X$ is $k$-regular as an $E$-module: it is generated in degrees $0 , \ldots,  -k$;  the first syzygies among these generators have potential degrees $-1 , \ldots,  -(k+1)$; and so on. In particular, if $X$ has maximal Albanese dimension -- i.e. when $k(X) = 0$ -- then $Q_X$ is generated in degree $0$ and has a linear resolution.
 
 The purpose of this note is to prove a more precise statement in case $k > 0$.
 \begin{theoremalpha}
 There is a canonical direct sum decomposition
\begin{equation}  Q_X \ = \ \bigoplus_{j = 0}^{k(X)} Q^j_X (j),\tag{*} \end{equation}
of $E$-modules, where $Q_X^j$ is $0$-regular.\footnote{Given a graded  $E$-module $M$, $M(j)$ denotes as usual the shifted $E$-module with $M(j)_\ell = M_{j + \ell}$.} Thus the minimal $E$-resolution of $Q_X$ is a direct sum of (shifts of) linear resolutions.  
 \end{theoremalpha}
 \noi Note that we do not assert that $Q_X^j \ne 0$ for every $j$, although $Q_X^k$ is necessarily non-vanishing. We remark that the existence of a direct sum decomposition (*) follows immediately from Koll\'ar's theorem \cite{kollar} on higher direct images of dualizing sheaves. The essential assertion of the Theorem is the regularity of the summands. Koll\'ar's decomposition appears in a related context in \cite{ch}. 
  
As in \cite{lp} the Theorem is proved by combining a package of results surrounding generic vanishing theorems with the BGG correspondence relating modules over an exterior algebra to linear complexes over a symmetric algebra.  
The additional tool required for the present improvement is a homological argument, based on a lemma on the degeneration of the spectral sequence associated to a filtered complex with homogeneous differentials going back to Deligne \cite{deligne}. Formalizing this in 
 the derived category, and combining it with Koll\'ar's decomposition theorem, allows one to circumvent some delicate questions of compatibility (as in \cite{ch}) between the derivative complexes appearing in \cite{gl2} and elsewhere, and the Leray spectral sequence.  
This background material is reviewed in \S 1.  
 The proof of Theorem A appears in \S 2. Finally, in \S 3 we make a few remarks concerning the extension of these ideas to more general integral transforms. 
 
 We are grateful to Herb Clemens, Christopher Hacon and Beppe Pareschi for helpful discussions, and to the referee for pointing out an  
 error in an earlier version.
 
 It is with respect and sadness that we dedicate this paper to the memory of Eckart Viehweg.   Viehweg's impact on the field was huge, and each of us has profited from his mathematics and his vision.  The first author in particular treasured thirty years of friendship with Eckart, from whom he learned so much.   Eckart will be  greatly missed.

\section{Preliminaries}

\noindent
{\bf Basics on the BGG correspondence.}
We briefly recall from \cite{efs}, \cite{eisenbud} and \cite{lp} some basic facts concerning the BGG correspondence.
Let $V$ be a $q$-dimensional complex vector space over a field $k$, and let 
$E = \bigoplus_{i=0}^q \bigwedge^i V$ be the exterior algebra over $V$. Denote by 
$W = V^\vee$ be the dual vector space, and by $S = {\rm Sym}(W)$ the symmetric algebra over $W$.       Elements of $W$ are taken to have degree $1$, while those in $V$ have degree $-1$.

Consider now a finitely generated graded module  $P = \bigoplus_{i=0}^dP_i$  over $E$. The \emph{dual} over $E$ of the module $P$ is defined to be the $E$-module 
\[ Q\ = \  \widehat{P} \ = \ \bigoplus_{j=0}^d P_{-j}^\vee \] (so that positive degrees are switched to negative ones and vice versa).  The basic idea of the BGG correspondence is that the properties of $Q$ as an $E$-module are controlled by a linear complex of $S$-modules constructed  from $P$. Specifically, one considers the complex $\L(P)$ given by
$$\cdots \lra S \otimes_{\CC} P_{j+1} \lra S\otimes_{\CC} P_j \longrightarrow S\otimes_{\CC} P_{j-1}\lra \cdots$$ 
with differential induced by 
$$s\otimes p \mapsto \sum_i x_i s \otimes e_i p,$$ 
where $x_i \in W$ and $e_i \in V$ are dual bases. 
We refer to \cite{efs} or \cite{eisenbud} for a dictionary linking $\L(P)$ and $Q$. 

As in \cite{lp}, we consider a notion of regularity for $E$-modules analogous to the theory of Castelnuovo-Mumford regularity for finitely generated $S$-modules, limiting ourselves here to  modules concentrated in non-positive degrees. 
\begin{definition} [\cite{lp} Definition 2.1] \label{regular} 
A finitely generated graded $E$-module $Q$ with no component of positive degree is called \emph{$m$-regular} if it is generated in degrees $0$ up to $-m$, and if its minimal free resolution has at most $m+1$ 
linear strands. Equivalently, $Q$   is $m$-regular if 
and only if 
\[ {\rm Tor}_i^E (Q, k)_{-i-j} \ = \ 0\]  for all $i \ge 0$ and all $j \ge m+1$.
\end{definition}
\noi In particular $0$-regular means being generated in degree $0$ and having a linear free $E$-resolution. The regularity of $Q$ can be computed from the BGG complex of  its dual as follows:

\begin{proposition}[\cite{lp} Proposition 2.2]\label{regularity}
Let $P$ be a  finitely generated graded module over $E$ with no component of negative
degree, say $P = \bigoplus_{i=0}^d P_i$.   Then $Q = \widehat{P}$ is $m$-regular if and only if $\L (P)$ is exact at the first $d-m$ steps from the left, i.e. if and only if the sequence
\[
0 \lra S \otimes_{\CC} P_{d} \lra S\otimes_{\CC} P_{d-1}\longrightarrow \cdots \lra S \otimes_{\CC} P_{m}\]
of $S$-modules is exact.  \qed
\end{proposition}

\noindent
{\bf Filtered complexes and BGG complexes.}
Let $(R, \frakm)$ be a regular local $k$-algebra of dimension $e$, with residue field $k
= R / \frakm$.  Let $W = \frakm / \frakm^2$ be the cotangent space, and $V = W^\vee$ the 
dual tangent space. We have $\frakm^p / \frakm^{p+1} \simeq {\rm Sym}^p W$ for all $p \geq 0$. 

Let $\bigl( K^{\bullet}, d \bigr)$ be a bounded complex of finitely generated free
$R$-modules. We can filter the complex by defining $F^p K^n = \frakm^p K^n$ for all $p \geq 0$ and all $n$; 
we then have 
$$F^p K^n / F^{p+1} K^n \simeq \bigl( K^n \otimes_R k \bigr) \otimes_k {\rm Sym}^p W.$$
The standard cohomological spectral sequence associated to a filtered complex then looks as follows
\begin{equation} \label{eq:SS-mm}
	E_1^{p,q} = H^{p+q} \bigl( K^{\bullet} \otimes_R k \bigr) \otimes_k {\rm Sym}^p W
		\Longrightarrow H^{p+q} \bigl( K^{\bullet} \bigr).
\end{equation}

Note that, as $V \cong {\rm Ext}_R^1 (k, k)$, there is a natural action for each $n$ given by
$$ V \otimes H^n \bigl( K^{\bullet} \otimes_R k \bigr) \longrightarrow H^{n+1} \bigl( K^{\bullet} \otimes_R k \bigr).$$
Consider now
$$P_{K^{\bullet}} : = \bigoplus_i H^i \bigl( K^{\bullet} \otimes_R k \bigr) 
{\rm~~and~~} Q_{K^{\bullet}} : = \bigoplus_i H^i \bigl( K^{\bullet} \otimes_R k \bigr)^\vee. $$
Assume (after shifting) that $P_{K^{\bullet}}$ lives in degrees $0$ up to $d$, where $d$ is a positive integer, and that 
the degree of its $H^i$ piece is $d-i$. Using the action above, we can then see
$P_{K^{\bullet}}$ and $Q_{K^{\bullet}}$ as dual finitely generated graded modules
over the exterior algebra $E = \bigoplus \bigwedge^i  V$ if we consider the ${H^i}^\vee$ piece in $Q_{K^{\bullet}}$ in degree $-i$; $Q_{K^{\bullet}}$ then lives in degrees $-d$ to $0$. 

One can apply the BGG correspondence described above to $P_{K^{\bullet}}$ and obtain a linear complex $\L (P_{K^{\bullet}})$ of finitely generated free $S$-modules, where $S = {\rm Sym}~ W$.

\begin{lemma}\label{e1_page}
The total complex of the $E_1$-page in \eqref{eq:SS-mm}, with terms
$$E_1^n = \bigoplus_{p+q=n} E_1^{p,q} = H^n \bigl( K^{\bullet} \otimes_R k \bigr) \otimes_k S,$$
is isomorphic to the complex $\L (P_{K^{\bullet}})$.
\end{lemma}
\begin{proof}
This is an extension, with the same argument, of \cite{lp} Lemma 2.3.
\end{proof}

\noi This gives in particular the following criterion for the vanishing of cohomology modules.

\begin{corollary}\label{abstract_vanishing}
Assume that the BGG complex $\L (P_{K^{\bullet}})$ is exact at the term $ H^n \bigl( K^{\bullet} \otimes_R k \bigr) \otimes_k S$. 
Then we have $H^n (K^{\bullet}) = 0$.
\end{corollary}
\begin{proof}
The exactness in the hypothesis is equivalent to the fact that $E^{p,q}_2 = 0$ for all $p$ and $q$ 
with $p+q = n$ in the spectral sequence \eqref{eq:SS-mm}, which implies the conclusion.
\end{proof}

\noindent
{\bf Complexes with homogeneous differentials.}
Consider as above  a bounded complex $\bigl( K^{\bullet}, d \bigr)$ of finitely generated free
$R$-modules.  Each differential $d^n \colon K^n \to K^{n+1}$ can
be viewed as a matrix with entries in the ring $R$. We say that the complex has
\emph{linear differentials} if there is a system of parameters $t_1, \dotsc, t_e$
for $R$ such that the entries of each $d^n$ are linear forms in $t_1, \dotsc,
t_e$. More generally, we say that $K^{\bullet}$ has \emph{homogeneous differentials
of degree $r$} if the entries of $d$ are homogeneous forms of degree $r$. Using the
fact that $R \subseteq k[[t_1, \ldots,t_e]]$, this means that the entries of each $d^n$, when viewed as
elements of the power series ring, are homogeneous polynomials of degree $r$.
The following degeneration criterion will be used in the proof of the main theorem.

\begin{lemma}\label{degeneration}
If $\bigl( K^{\bullet}, d \bigr)$ has homogeneous differentials of degree $r$, then the spectral
sequence in \eqref{eq:SS-mm} degenerates at the $E_{r+1}$-page. In particular, it degenerates
at the $E_2$-page if the complex has linear differentials.
\end{lemma}
\begin{proof}
We use the technical criterion in Lemma~\ref{lem:degen} below: the spectral sequence
\eqref{eq:SS-mm} degenerates at the $E_{r+1}$-page if, and only if, the filtration satisfies
$$ \frakm^k K^n \cap d \bigl( K^{n-1} \bigr) \subseteq d \bigl( \frakm^{k-r} K^{n-1} \bigr)$$
for all $n, k \in \ZZ$. Let
$t_1, \dotsc, t_e$ be a system of parameters such that each differential $d$ in the
complex is a matrix whose entries are homogeneous polynomials of degree $r$. Now
suppose we have a vector $x \in K^{n-1}$ such that $dx$ has entries in $\frakm^k$. Write $x = x' +
x''$, in such a manner that the components of $x'$ are polynomials of degree at most
$k-r-1$, while the components of $x''$ belong to $\frakm^{k-r}$. All components of the
vector $dx'$ are then polynomials of degree at most $k-1$ in $t_1, \dotsc, t_e$,
while those of $dx''$ belong to $\frakm^k$.  The fact that $dx = dx' + dx''$ also has
entries in $\frakm^k$ then forces $dx' = 0$; thus $dx = dx''$. We conclude by the criterion mentioned above.
\end{proof}

The following degeneration criterion for the spectral sequence of a filtered complex
generalizes \cite{deligne} \S1.3, where the case $r=0$ is proved.

\begin{lemma}\label{lem:degen}
The spectral sequence of a filtered complex $(K^{\bullet}, F^{\bullet})$ degenerates at the
$E_{r+1}$-page if, and only if, the filtration satisfies
\begin{equation} \label{eq:cond-deg}
	F^k K^n \cap d \bigl( K^{n-1} \bigr) \subseteq d \bigl( F^{k-r} K^{n-1} \bigr)
\end{equation}
for all $n, k \in \ZZ$.
\end{lemma}
\begin{proof}
By \cite{deligne} \S1.3, the entries of the spectral sequence are given by
\[
	E_r^{p,q} = \im \bigl( Z_r^{p,q} \to K^{p+q} / B_r^{p,q} \bigr),
\]
where we have set 
\begin{align*}
	Z_r^{p,q} &= \ker \bigl( d \colon F^p K^{p+q} \to
		K^{p+q+1} / F^{p+r} K^{p+q+1} \bigr), \\
	K^{p+q} / B_r^{p,q} &= \coker \bigl( d \colon F^{p-r+1} K^{p+q-1} \to
		K^{p+q} / F^{p+1} K^{p+q} \bigr).
\end{align*}
Moreover, the differential $d_r$ is defined by the rule $d_r [x] = [d x]$ for $x \in
Z_r^{p,q}$.

To prove the criterion, let us first assume that the spectral sequence degenerates at
$E_{r+1}$; in other words, that $d_{r+1} = d_{r+2} = \dotsb = 0$. This means that 
\[
	E_{i+1}^{p,q} = E_{\infty}^{p,q} =
		F^p H^{p+q} \bigl( K^{\bullet} \bigr) / F^{p+1} H^{p+q} \bigl( K^{\bullet} \bigr)
\]
for all $p,q \in \ZZ$ and $i \geq r$. Thus for any $x \in F^p K^{p+q}$ that satisfies $dx \in
F^{p+i+1} K^{p+q+1}$, there exists some $y \in F^p K^{p+q}$ with $dy = 0$ and
$x - y \in F^{p+1} K^{p+q} + d \bigl( F^{p-i+1} K^{p+q-1} \bigr)$. In other words,
taking $n = p+q+1$ and $k = p + i + 1$, we have
\[
	F^k K^n \cap d \bigl( F^{k-i-1} K^{n-1} \bigr) \subseteq 
		d \bigl( F^{k-i} K^{n-1} \bigr).
\]
Since the filtration is exhaustive, \eqref{eq:cond-deg} follows by descending induction on $i \geq r$.

Conversely, suppose that the condition in \eqref{eq:cond-deg} is satisfied. Consider
an arbitrary class $[x] \in E_{r+1}^{p,q}$, represented by an element $x \in
F^p K^{p+q}$ with $dx \in F^{p+r+1} K^{p+q+1}$. Since $F^{p+r+1} K^{p+q+1} \cap d
\bigl( K^{p+q} \bigr) \subseteq d \bigl( F^{p+1} K^{p+q} \bigr)$ by virtue of
\eqref{eq:cond-deg}, we can find $z \in F^{p+1} K^{p+q}$ with $dx = dz$. But then $y
= x - z$ represents the same class as $x$ and satisfies $dy = 0$; this shows that
$d_i [x] = d_i [y] = 0$ for all $i \geq r+1$, and proves the degeneracy
of the spectral sequence at $E_{r+1}$.
\end{proof}

\begin{example}
Here is an example showing that Lemma~\ref{degeneration} does not necessarily hold 
when the entries of the matrices representing the differentials are polynomials of degree \emph{at
most} $r$, hence homogeneity is necessary. 
Let $R = k[\![t]\!]$ be the ring of power series in one variable, and
consider the complex of free $R$-modules
\begin{diagram}[width=3em]
	R^{\oplus 2} 
		&\rTo^{ A = \left( \begin{smallmatrix} 1 & -t \\ t & 0 \end{smallmatrix} \right)}& 
		R^{\oplus 2}.
\end{diagram}
Now take $x = (t, 1)$; then $A\cdot x = \bigl( 0, t^2 \bigr)$ has components in
$\mathfrak{m}^2$, but there is no vector $y$ with components in $\mathfrak{m}$ such
that $A\cdot y = \bigl( 0, t^2 \bigr)$. Since the condition in Lemma~\ref{lem:degen} is
violated, the spectral sequence does not degenerate at $E_2$ in this case.
\end{example}

\noindent
{\bf Passing to the derived category.}
We now generalize the construction of the spectral sequence to
arbitrary objects in the bounded derived category $\D(R)$ of finitely generated
$R$-modules. We use the fact that $\D(R)$ is isomorphic to the category of bounded
complexes of finitely generated free $R$-modules up to homotopy \cite{weibel} \S10.4.

Let $C^{\bullet}$ be any object in $\D(R)$. We can find a complex $K^{\bullet}$ of
finitely generated free $R$-modules quasi-isomorphic to $C^{\bullet}$. If we apply
the construction of the previous paragraph to $K^{\bullet}$, we obtain a
spectral sequence
\begin{equation} \label{eq:SS-mm-gen}
	E_1^{p,q} = H^{p+q} \bigl( C^{\bullet} \otimes_R k \bigr) \otimes_k {\rm Sym}^p W
		\Longrightarrow H^{p+q} \bigl( C^{\bullet} \bigr),
\end{equation}
in which both the $E_1$-page and the limit can be computed from $C^{\bullet}$.

It remains to show that the spectral sequence \eqref{eq:SS-mm-gen} is well-defined and
only depends on $C^{\bullet}$ up to isomorphism in $\D(R)$. It suffices to show that if 
$C^{\bullet} \rightarrow D^{\bullet}$ is a quasi-isomorphic map of complexes, then it induces 
a canonical isomorphism between the two associated spectral sequences. Let $L^{\bullet}$
be a bounded complex of finitely generated free $R$-modules quasi-isomorphic to
the complex $D^{\bullet}$. Then $K^{\bullet}$ and $L^{\bullet}$ are also quasi-isomorphic, and
since both are bounded complexes of free $R$-modules, they must be homotopy
equivalent (see \cite{weibel} Theorem 10.4.8). In other words, there are maps of complexes
\[
	f \colon K^{\bullet} \to L^{\bullet} \qquad \text{and} \qquad
		g \colon L^{\bullet} \to K^{\bullet}
\]
such that $f \circ g - \operatorname{id}_{L^{\bullet}}$ and $g \circ f -
\operatorname{id}_{K^{\bullet}}$ are null-homotopic. Moreover, any two choices of $f$
(resp.\@ $g$) are homotopic to each other.  Since $f$ and $g$ clearly preserve the
filtrations defined by powers of $\frakm$, we get induced maps 
\[
	F_r \colon E_r^{p,q}(K^{\bullet}) \to E_r^{p,q}(L^{\bullet}) \quad \text{and}
		\quad
	G_r \colon E_r^{p,q}(L^{\bullet}) \to E_r^{p,q}(K^{\bullet})
\]
between the spectral sequences for the two filtered complexes. 

\begin{lemma} \label{lem:homotopy}
Let $h \colon K^{\bullet} \to L^{\bullet}$ be a map between two bounded complexes of
free $R$-modules. If $h$ is null-homotopic, then the induced map on spectral
sequences is zero.
\end{lemma}
\begin{proof}
Since $h \sim 0$, there is a collection of maps $s^n \colon K^n \to L^{n-1}$ with the
property that $h^n = d \circ s^n - s^{n+1} \circ d$. It follows that the maps $H^n
\bigl( K^{\bullet} \otimes_R k \bigr) \to H^n \bigl( L^{\bullet} \otimes_R k \bigr)$
induced by $h$ are zero. But this means that the map of spectral sequences induced by
$h$ is also zero, starting from the $E_1$-page.
\end{proof}

Since any two choices of $f$ are homotopic to each other, it follows that the maps
$F_r$ are independent of the choice of $f$, and therefore canonically determined by
the two complexes $K^{\bullet}$ and $L^{\bullet}$. Moreover, the fact that $f \circ
g$ and $g \circ f$ are homotopic to the identity shows that $F_r \circ G_r$ are $G_r
\circ F_r$ are equal to the identity. This proves that the two spectral sequences are
canonically isomorphic. 

\begin{corollary} \label{cor:E2}
Let $C^{\bullet}$ be any object in $\D(R)$. If $C^{\bullet}$ is quasi-isomorphic to
a complex of free $R$-modules with linear differentials, then the spectral sequence
\eqref{eq:SS-mm-gen} degenerates at the $E_2$-page.
\end{corollary}

\begin{corollary} \label{cor:sum}
Let $C_i^{\bullet}$, for $i = 1, \dotsc, k$, be a collection of objects in $\D(R)$.
If $C^{\bullet} = C_1^{\bullet} \oplus \ldots \oplus C_k^{\bullet}$ is quasi-isomorphic to a
complex of free $R$-modules with linear differentials, then the spectral sequence 
\eqref{eq:SS-mm-gen} for each $C_i^{\bullet}$ degenerates at the $E_2$-page.
\end{corollary}
\begin{proof}
Since the spectral sequence for $C^{\bullet}$ is the direct sum of the individual
spectral sequences, the assertion follows immediately from Corollary \ref{cor:E2}. 
\end{proof}

\section{Decomposition of the canonical cohomology module}

Let $X$ be a smooth projective complex variety $X$ of dimension $d$ and irregularity $q = h^1 (X, \OO_X)$, with
\[  a : X \lra \Alb(X)\] the Albanese mapping of $X$,
and let 
\[ k \ = \ k(X) \ = \ \dim X \,  - \, \dim a(X) \]
be the dimension of the general fiber of $a$. Define 
\[ P_X \, = \, \bigoplus_{i = 0}^d ~ \HH{i}{X}{ \OO_X} \ \ , \ \ 
Q_X \, = \, \bigoplus_{i=0}^d~ H^i (X, \omega_X),
\] 
These are dual graded modules over 
the exterior algebra $E = \bigoplus_{i=0}^q \bigwedge^i V$, with $V = H^1 (X, \OO_X)$. Here $H^i (X, \omega_X)$ is considered in degree 
$-i$ and $H^i (X, \OO_X)$ in degree $d-i$.
The main result of Koll\'ar \cite{kollar} asserts that one has a splitting
\begin{equation}\label{eq:splitting}
\RR a_* \omega_X \ \cong \ \bigoplus_{j = 0}^k \, R^ja_* \omega_X\, [-j] 
\end{equation}
in the derived category $\D(A)$. Therefore $Q_X$ can be expressed as a direct sum
\[
Q_X \ = \ \bigoplus_{j=0}^k\,   Q_X^j (j) \ \ ,  \ \ \text{with} \ \  Q_X^j  \ = \
\bigoplus_{i=0}^d \HH{i}{A}{R^ja_* \omega_X}.\]
Moreover this is a decomposition of $E$-modules: $E$ acts on $\HH{*}{A}{R^ja_* \omega_X}$ via cup product through the identification $\HH{1}{X}{\OO_X} = \HH{1}{A}{\OO_A}$, and we again consider $\HH{i}{A}{R^ja_* \omega_X}$ to live in degree $-i$.
In \cite{lp} Theorem B it was proved that the regularity of $Q_X$ over $E$ is equal to $k$. Here we prove the stronger

\begin{theorem}\label{canonical}
The modules $Q_X^j$ are $0$-regular for all $j= 0,\ldots,k$, and the minimal $E$-resolution of $Q_X$
splits into the direct sum of the linear resolutions of $Q_X^j (j)$. 
\end{theorem}
\begin{proof}
Let $P$ be a Poincar\'e bundle on $X \times \Pic0$. We denote by 
$$\RR\Phi_P : \D(X) \rightarrow \D({\rm Pic}^0 (X)),  \ \ \RR\Phi_P \E = \RR {p_2}_* (p_1^* \E \otimes P)$$
the integral functor given by $P$, and analogously for $\RR\Phi_{P^\vee}$. 
Following Mukai's notation \cite{mukai}, we also denote by 
$$\RR\widehat{\SS} : \D(A) \rightarrow \D({\rm Pic}^0 (X))$$ 
the standard Fourier-Mukai functor on $A$, again given by the respective Poincar\'e bundle $\P$, with $P = (a \times {\rm id})^* \P$.  We have $\RR \Phi_{P} = \RR\widehat{\SS} \circ \RR a_*$ and $\RR \Phi_{P^\vee} = (-1)^* \circ \RR\widehat{\SS} \circ \RR a_*$. 

For an object $\F$ we denote $\F^\vee = \R \mathcal{H}om (\F, \OO_X)$ and $\R \Delta \F := \R \mathcal{H}om (\F, \omega_X)$. Grothendieck duality gives for any object $\F$ in $\D(X)$ 
or $\D(A)$ the following formulas (see for instance the proof of \cite{pp2} Theorem 2.2):
\begin{equation}\label{eq:gd}
(\R\Phi_{P^\vee} \F)^\vee \simeq \R\Phi_P (\R\Delta \F)[d] \quad {\rm and} \quad (-1)^*(\R \widehat{\SS}\F)^\vee \simeq \R \widehat{\SS} (\F^\vee)[q].
\end{equation}
Applying this to $\F = \omega_X$, we obtain 
$$\RR \Phi_{P}  \OO_X  \ \cong \ \bigl( \RR \Phi_{P^\vee}  \omega_X \bigr)^\vee [-d]$$
and therefore by (\ref{eq:splitting}) it follows that in $\D (\Pic0)$ we have a splitting
\begin{equation} \label{eq:FM}
\begin{split}
\RR \Phi_{P}  \OO_X  \ &\cong \  \bigoplus_{j = 0}^k (-1)^* \bigl(\RR \widehat{\SS}
(R^ja_* \omega_X) \bigr)^\vee \ [j-d]  \\
\ &\cong \ \bigoplus_{j = 0}^k \RR \widehat{S} \bigl( (R^ja_* \omega_X)^\vee \bigr) [q+ j -d],
\end{split}
\end{equation}
where the second isomorphism again comes from (\ref{eq:gd}) applied to $\F = R^j a_* \omega_X$.
In addition, by \cite{hacon} \S4 we know that for all $j$ the sheaf $R^j a_* \omega_X$ satisfies the property
$$\RR \widehat{S} \bigl( (R^ja_* \omega_X)^\vee \bigr) \simeq R^q \widehat{S} \bigl( (R^ja_* \omega_X)^\vee \bigr)[-q],$$
meaning that the Fourier-Mukai transform of its dual is supported only in degree $q$.\footnote{As the referee points out, 
this is the only result used in the proof that at the moment is not known to hold in the K\"ahler setting.}
Combined with (\ref{eq:FM}), this finally gives the decomposition into a direct sum of (shifted) sheaves
\begin{equation}\label{eq:sum}
\RR \Phi_{P}  \OO_X  \ \cong \  \bigoplus_{j = 0}^k \  \G_j \ [j-d]
\end{equation}
with $\G_j : = R^q \widehat{\SS} \bigl( (R^ja_* \omega_X)^\vee \bigr)$.

We may pull back the object $\R \Phi_P \OO_X$ via the exponential map ${\rm exp}: V \rightarrow \Pic0$ centered at the origin. 
By \cite{gl2}, Theorem 3.2, we then have the identification of the analytic 
stalks at the origin
$$\mathcal{H}^i\big(({\mathcal{K}^{\bullet}})^{an}\big)_0  \  \cong  \ (R^i {p_2}_* \P)_0.$$
where $(\mathcal{K}^{\bullet})^{an}$ is the complex of trivial analytic vector bundles on $V$:
\[
0 \lra \OO_V \otimes \HH{0}{X}{\OO_X} \lra \OO_V \otimes \HH{1}{X}{\OO_X} \lra \ldots \lra \OO_V \otimes \HH{d}{X}{\OO_X} \lra 0,
\]
with maps given at each point of $V$  by wedging with the corresponding element of $\HH{1}{X}{\OO_X}$.
Passing from analytic to algebraic sheaves as in \cite{lp} Proposition 1.1, we obtain that the 
the stalk of $\R \Phi_P \OO_X$ at the origin is quasi-isomorphic to a complex of free modules over $R = \OO_{\Pic0, 0}$
whose differentials are linear (with respect to a system of parameters corresponding to a euclidean coordinate system on $V$ via the exponential map). This implies by Lemma \ref{degeneration} that the
spectral sequence derived from (1.2)
$$E_1^{p,q} = H^{p+q} \bigl({(\R \Phi_P \OO_X)}_0 \otimes_R k\bigr) \otimes_k {\rm Sym}^p W 
\Longrightarrow {(R^{p+q} \Phi_P \OO_X)}_0$$
degenerates at the $E_2$-page. By Lemma \ref{e1_page}, the total $E_1$-page of this spectral sequence is the BGG 
complex $\L (P_X)$. Denoting by $P_X^j$ the $E$-module dual to $Q_X^j$, the functoriality of the BGG correspondence implies that 
we have a decomposition
$$\L(P_X) \ \simeq \ \bigoplus_{j=0}^k\,  \L(P_X^j) [j].$$
Fix an index $0 \leq j \leq k$. Given (\ref{eq:sum}),  Corollary~\ref{cor:sum} implies that for each $j$ the spectral sequence 
$$E_1^{p,q} = H^{p+q} (\G_j \otimes_R k) \otimes_k {\rm Sym}^p W \Longrightarrow H^{p+q} (\G_j)_0.$$
degenerates at the $E_2$-page as well, while its total $E_1$-page is the complex $\L(P_X^j)$. 
As $\G_j$ is a single $R$-module, the limit of the 
corresponding spectral sequence is $0$ except for $p +q =0$. Now using again
Lemma \ref{e1_page}, this means that each complex $\L \bigl( P_X^j \bigr)$ is
exact except at its right end. But by Proposition \ref{regularity} 
this is equivalent to the fact that each $Q_X^j$ is $0$-regular as a graded $E$-module.
\end{proof}

\begin{remark}
The result above gives in particular a proof of the following statement for all the $R^j a_* \omega_X$: 
the stalks of the Fourier-Mukai transforms of their derived duals admit filtrations whose associated graded modules
are the cohomologies of their (linear) derivative complexes at the origin. The argument given here circumvents the need to check the compatibility of the differentials in the Leray spectral sequence with those in the derivative complex, stated in \cite{ch} Theorem 10.1.
\end{remark}

\section{General integral transforms}

We conclude by briefly noting that the technical material in the previous sections can be formally extended from the setting of 
$\R\Phi_P$ in the proof of Theorem \ref{canonical} to that of arbitrary integral functors. We hope that this may have future 
applications. The main point, undoubtedly known to experts, is that the BGG complexes we consider can be seen as \emph{local linearizations} of integral transforms.

Let $X$ and $Y$ be smooth projective varieties over an algebraically closed field $k$, of dimensions $d$ and $e$ respectively, 
and let $P$ be a locally free sheaf on $X \times Y$\footnote{This can be easily extended to any coherent sheaf on $X\times Y$, flat over $Y$.} 
inducing the integral transform
$$\R \Phi_P : \D (X) \rightarrow \D(Y),
~~ \R \Phi_P (\cdot) : = \R {p_Y}_* ( p_X^*(\cdot) \overset{\L}\otimes P).$$
For any $y\in Y$ we denote $P_y := P_{|X\times\{y\}}$. 
Fix $y_0 \in Y$, and denote $(R, \frakm) = (\OO_{Y, y_0}, \frakm_{y_0})$.
For any object $\F$ in $\D(X)$, we can consider the object 
$$C_{\F}^{\bullet} :=  \L i_{y_0}^*(\R\Phi_P \F) {\rm  ~in~} \D(R).$$
Recall that we use the notation $\R \Delta \F := \R \mathcal{H}om (\F, \omega_X)$.

\noindent
{\bf Exterior module structure.}
For each $n$ we have a basic isomorphism
\begin{equation}\label{eq:basic}
{\mathbb H}^n (X, \F\otimes P_{y_0}) \cong H^n (C^{\bullet}\otimes_R k(y_0))
\end{equation}
by applying the Leray isomorphism and the projection formula (cf. \cite{pp1} Lemma 2.1). 
We consider
$$P_{C_{\F}^\bullet}  : = \bigoplus_n {\mathbb H}^n (X, \F \otimes P_{y_0}) {\rm~~and~~}
Q_{C_{\F}^\bullet}  : = \bigoplus_n {\mathbb H}^n (X, \R\Delta \F \otimes P_{y_0}^\vee)$$
as dual modules over the exterior algebra $E = \bigoplus \bigwedge^i T_{Y, y_0}$, via the Kodaira-Spencer map
$$\kappa_{y_0}: T_{Y, y_0} \longrightarrow {\rm Ext}^1 (P_{y_0}, P_{y_0})$$
and the natural cup-product action of ${\rm Ext}^1 (P_{y_0}, P_{y_0})$ on $P_{C_{\F}^\bullet}$.  

\medskip

\noindent
{\bf Comparison of the BGG complex with the integral transform.}
The spectral sequence in  \eqref{eq:SS-mm-gen} applied to the object $C_{\F}^{\bullet}$, together with the isomorphism 
in \eqref{eq:basic}, translates into

\begin{lemma}\label{new_ss}
There is a cohomological spectral sequence
$$E_1^{p,q} = {\mathbb H}^{p+q} (X, \F\otimes P_{y_0}) \otimes_k {\rm Sym}^p W
\Longrightarrow H^{p+q} ( C_{\F}^{\bullet} ).$$
\end{lemma}

By Lemma \ref{e1_page}, the total $E_1$-complex associated to this spectral sequence is the 
BGG complex $\L (P_{C_{\F}^\bullet})$.
This leads to the interpretation mentioned above: \emph{the BGG complex $\L (P_{C_{\F}^\bullet})$ is a linearization of the integral
transform $\R\Phi_P \F$ in a neighborhood of $y_0$}. 

This implies a refinement of the familiar base change criterion for the local vanishing of the 
higher derived functors, saying that if $H^n (X,\F \otimes P_{y_0}) = 0$, then $(R^n \Phi_P \F)_{y_0} = 0$. 

\begin{proposition}\label{vanishing}
Assume that the BGG complex $\L(P_{C_{\F}^\bullet})$ is exact at the term $H^n (X, \F\otimes P_{y_0}) \otimes_k S$. Then 
$(R^n \Phi_P \F)_{y_0} = 0$.
\end{proposition}
\begin{proof}
This follows from Lemma \ref{new_ss} and Proposition \ref{abstract_vanishing}.
\end{proof}

\begin{remark}
Note that the converse is true only if the spectral sequence degenerates at the
$E_2$-term, in particular if $C_{\F}^{\bullet}$ can be represented around $y_0$ by a
complex of free $R$-modules with linear differentials. It would be
interesting to study the differentials $d_i$ in the spectral sequence for $i \geq 2$.
\end{remark}

Let now $\G$ be a sheaf on $X$, and $\F := \R\Delta \G$. Note that in this case $P_{C_{\F}^\bullet}$ will live in degrees 
$0$ to $d = \dim X$ and $Q_{C_{\F}^\bullet}$ in degrees $-d$ to $0$.

\begin{corollary}[BGG exactness implies generic vanishing]\label{gv}
Assume that $\L(P_{C_{\F}^\bullet})$ is exact at the first $d-k$ 
steps from the left, or equivalently that $Q_{\F, y_0}$ is $k$-regular over $E$. Then  
$${\codim}_{y_0}{\rm Supp}~R^i \Phi_{P^\vee} \G \ge i -k, {\rm~for ~all~}i >0.$$
(In the language of \cite{pp1}, $\G$ is a $GV_{-k}$-sheaf with respect to $P^\vee$ in a neighborhood of $y_0$.)
\end{corollary}
\begin{proof}
Proposition \ref{vanishing} implies that $(R^i \Phi_P \F)_{y_0} = 0$ for all $i < d-k$. By the basic local $GV$-$WIT$ equivalence 
\cite{pp2} Theorem 2.2, this is equivalent to the fact that $\G$ is $GV_{-k}$ with respect to $P^\vee$.
\end{proof}

\providecommand{\bysame}{\leavevmode\hbox
to3em{\hrulefill}\thinspace}

\end{document}